\documentclass{amsart}
\usepackage{amssymb,amscd,amsmath,amsthm}

\theoremstyle{plain}
\newtheorem{thm}{Theorem}[section]

\newtheorem{lem}[thm]{Lemma}

\theoremstyle{definition}

\theoremstyle{remark}

\newtheorem*{remark*}{Remark}

\numberwithin{equation}{section}

\begin{document}

\title[On a problem of geometry of numbers arising
in spectral theory]{On a problem of geometry of numbers arising in spectral theory}

\author[Y. A. Kordyukov]{Yuri A. Kordyukov}
\address{Institute of Mathematics\\
         Russian Academy of Sciences\\
         112~Chernyshevsky str.\\ 450008 Ufa\\ Russia} \email{yurikor@matem.anrb.ru}

\author[A. A. Yakovlev]{Andrey A. Yakovlev}

\address{Institute of Mathematics\\
         Russian Academy of Sciences\\
         112~Chernyshevsky str.\\ 450008 Ufa\\ Russia}

\email{yakovlevandrey@yandex.ru}



\begin{abstract}
We study a lattice point counting problem for a class of families of domains in a Euclidean space. This class consists of anisotropically expanding bounded domains, which remain unchanged along some fixed linear subspace and expand in directions, orthogonal to this subspace. We find the leading term in the asymptotics of the number of lattice points in such  family of domains and prove remainder estimates in this asymptotics under various conditions on the lattice and the family of domains.

As a consequence, we prove an asymptotic formula for the eigenvalue distribution function of the Laplace operator on a flat torus in adiabatic limit determined by a linear foliation with a nontrivial remainder estimate.
\end{abstract}

\dedicatory{Dedicated to the memory of Arlen Mikhailovich Ilyin}

\thanks{Supported by the Russian Foundation of Basic Research
(grant no. 12-01-00519-a)}

\date{}

\maketitle


\section{Introduction}

\subsection{Statement of the problem}\label{s:setting} The paper is devoted to lattice point counting problems for some families of anisotropically expanding domains in a Euclidean space. More precisely,  consider a finite-dimensional
real vector space $E$ endowed with an inner product $(\cdot,\cdot)$.
Let $F$ be a $p$-dimensional linear subspace of $E$ and $H=F^{\bot}$
the $q$-dimensional orthogonal complement of $F$, $p+q=n:=\dim E$.
For any $\varepsilon>0$, define a linear transformation $T_\varepsilon :
E\to E$ by
\begin{equation}\label{e:Te}
T_\varepsilon(x)=\begin{cases} x, & \text{if}\ x\in F, \\
\varepsilon^{-1}x, & \text{if}\ x\in H.
\end{cases}
\end{equation}

Let $\Gamma\subset E$ be a lattice in $E$. For any bounded open subset
$S$ in $E$, we denote by $n_\varepsilon (S)$ the number of points of
$\Gamma$ in the set $T_\varepsilon(S)$:
\begin{equation}\label{e:nS}
n_\varepsilon (S)=\# (T_\varepsilon(S)\cap \Gamma), \quad
\varepsilon>0.
\end{equation}

The main purpose of the paper is to study the asymptotic behavior of
$n_\varepsilon (S)$ as $\varepsilon\to 0$. In the particular case when
$E$ is the linear space ${\mathbb R}^n$ equipped with the standard Euclidean structure and
$\Gamma={\mathbb Z}^n$, this problem was studied in
\cite{lattice-points,mz-varchenko,varchenko}. Lattice
point counting problems for families of anisotropically expanding domains were
also investigated in considerable detail in the papers
\cite{Skriganov89,Skriganov94,Nikichine-Skriganov95,Nikichine-Skriganov98}. These papers are
devoted to the study of anomalously small, particularly
logarithmically small, errors in the lattice point counting problem (see also
the introduction of \cite{lattice-points}).

\subsection{Adiabatic limits}\label{s:adiab1} 
The problem mentioned above appears naturally in the study of the following asymptotic problem in spectral theory of differential operators.

Consider a flat torus $\mathbb T=\mathcal E/\Lambda$,
where $\mathcal E$ is a real $n$-dimen\-si\-o\-nal Euclidean space equipped with a Euclidean metric $g$ and $\Lambda$ is a lattice in 
$\mathcal E$. Let $F$ be a linear subspace in $\mathcal E$ and $\mathcal F$ be the corresponding linear foliation on $\mathbb T$: the leaf
$L_x$ of $\mathcal F$ through a point $x\in \mathbb T$ has the form:
\[
L_x=x+F \mod \Lambda.
\]
The direct sum decomposition $
\mathcal E=F\bigoplus H$, $H:=F^\bot$, of $\mathcal E$  induces the corresponding decomposition $g=g_{F}+g_{H}$ of the metric $g$. We define a one-parameter family $g_{\varepsilon}$ of Euclidean metrics on $\mathcal E$ by
\[
g_{\varepsilon}=g_{F} + {\varepsilon}^{-2}g_{H}, \quad \varepsilon >
0.
\]
We will also consider $g_{\varepsilon}$ as a Riemannian metric on  $\mathbb{T}$.

For any $\varepsilon>0$, consider the Laplace-Beltrami operator 
$\Delta_\varepsilon$ on $C^\infty(\mathbb T)$ associated with $g_\varepsilon$. In the linear coordinates 
$(x_1,x_2,\ldots,x_n)$ on $\mathbb{T}$ given by a basis
$(e_1,e_2,\ldots,e_n)$ in $\mathcal E$, the operator 
$\Delta_\varepsilon$ is written as 
\[
\Delta_\varepsilon =-\sum_{j,\ell=1}^n g_\varepsilon^{j\ell}
\frac{\partial^2}{\partial x_j\partial x_\ell},
\]
where $g_\varepsilon^{j\ell}$ are the elements of the inverse matrix of the matrix of the Euclidean metric $g_\varepsilon $ in the basis $(e_1,e_2,\ldots,e_n)$.

The problem is to study the asymptotic behavior of eigenvalues of the operator $\Delta_\varepsilon$ as $\varepsilon$. Such a limiting procedure for the eigenvalue problem for the Laplace operator on a compact Riemannian manifold is often called the passage to adiabatic limit. In the geometric setting, the notion of adiabatic limit was introduced by Witten for the Dirac operator on the total space of a fiber bundle over the circle in 1985 in the study of global anomalies in string theory. In quantum mechanics, adiabatic limits are closely related with so-called Born-Oppenheimer approximation. We refer the reader to a survey paper \cite{bedlewo2-andrey} for historic remarks and references.

Let $\mathcal E^*$ denote the dual Euclidean space of $\mathcal E$ and $\Lambda^*\subset \mathcal E^*$ the dual lattice of $\Lambda$:
\[
\Lambda^*=\{k\in \mathcal E^* : \langle k,\Lambda\rangle
\subset\mathbb Z\}.
\]
The operator $\Delta_\varepsilon$ has a complete orthonormal system of eigenfunctions 
\[
U_{k}(x)=e^{2\pi i \langle k,x\rangle }, \quad x\in \mathbb T,\quad
k\in \Lambda^*\subset \mathcal E^*,
\]
with the corresponding eigenvalues 
\[
\lambda_{k}=4\pi^2 |k|^2_{g_\varepsilon^{-1}} =4\pi^2
|T^{-1}_\varepsilon(k)|^2_{g^{-1}},
\]
where $g_\varepsilon^{-1}$ is the metric on $\mathcal E^*$ induced by $g_\varepsilon$, $T_\varepsilon$ is the linear operator on $\mathcal E^*$ defined by the decomposition $ \mathcal E^*=F^*\bigoplus H^*$ via the formula \eqref{e:Te}.

Therefore, the eigenvalue distribution function $N_\varepsilon(\lambda)$ of $\Delta_\varepsilon$ defined by 
\[
N_\varepsilon(\lambda)=\sharp\{k\in \Lambda^* : \lambda_{k}<\lambda
\}, \quad \lambda \in {\mathbb R},
\]
is described as
\begin{equation}
\label{e:Ne} N_\varepsilon (4\pi^2\lambda)=n_\varepsilon
(B_{\sqrt{\lambda}}(0)), \quad \lambda \in {\mathbb R},
\end{equation}
where $n_\varepsilon (B_{\sqrt{\lambda}}(0))$ is the number of points of the lattice $\Lambda^*$ in the ellipsoid $T_\varepsilon(B_{\sqrt{\lambda}}(0))$ (cf. \eqref{e:nS}). In particular, the problem on the asymptotic behavior of $N_\varepsilon(\lambda)$ as $\varepsilon\to 0$ is equivalent to the problem formulated in Section~\ref{s:setting} in the case when 
$S$ is the ball $B_{\sqrt{\lambda}}(0)$.

In~\cite{adiab} (see also \cite{asymp}) the  first author studied the asymptotic behavior of the eigenvalue distribution function
$N_\varepsilon(\lambda)$ of the Laplace operator on a compact Riemannian manifold $M$ equipped with a foliation $\mathcal F$ in adiabatic limit. He computed the leading term of asymptotics of $N_\varepsilon(\lambda)$ as $\varepsilon\to 0$ in the case when the foliation $\mathcal F$ is Riemannian and the metric $g$ is bundle-like. A linear foliation on a torus is a Riemannian foliation, and a Euclidean metric on the torus is bundle-like. Therefore, one can apply the results of  \cite{adiab,asymp} and obtain the leading term of asymptotics of  $N_\varepsilon(\lambda)$ in the case of a linear foliation on a torus. The remainder estimates in the asymptotic formula for the number of lattice points $n_\varepsilon (S)$ obtained in this paper (see Section \ref{s:as} below) allow us to prove a nontrivial remainder estimate in the asymptotic formula for $N_\varepsilon(\lambda)$ in this particular case (see Section \ref{s:adiab0} below).

\section{Main results}

In this section, we formulate the main results of the paper.
We use notation of Section~\ref{s:setting}

\subsection{Asymptotic formula and remainder estimates}\label{s:as}
First of all, we describe the leading term of asymptotics of 
$n_\varepsilon (S)$ as $\varepsilon\to 0$ under very general assumptions on $S$.

Put $\Gamma_F = \Gamma^*\cap F$, where $\Gamma^*\subset E$ denotes the dual lattice of $\Gamma$:
\[
\Gamma^*=\{ \gamma^*\in E :(\gamma^*,\Gamma)\subset\mathbb Z\}.
\]
It is clear that $\Gamma_F$ is a free abelian group. Let
$V$ be the linear subspace of $E$ spanned by the elements of $\Gamma_F$, $r=\dim V$. Note that $\Gamma_F$ is a lattice in 
$V$. Let $\Gamma^*_F\subset V$ be the dual lattice of $\Gamma_F$.

Let $V^\bot$ be the orthogonal complement of $V$. One can show that 
\[
\Gamma^\bot=\Gamma\bigcap V^\bot
\]
is a lattice in $V^\bot$. For any $x\in V$, denote by 
$P_{x}$ the $(n-r)$-dimensional affine subspace in $E$ passing through $x$ orthogonal to $V$:
\[
P_x=x+V^\bot.
\]

\begin{thm}\label{t:lattice_points-basic} For any bounded subset 
$S$ of $E$ such that, for any $\gamma^*\in\Gamma^*_F$, the intersection $P_{\gamma^*} \cap S$ is an open, Jordan measurable subset in $P_{\gamma^*}$, the following asymptotic formula holds:
\begin{equation}\label{f:lattice_points-basic}
 n_\varepsilon (S) \sim
\frac{\varepsilon^{-q}}{{\rm vol}(V^\bot/\Gamma^\bot)}
 \sum_{\gamma^*\in\Gamma_F^*} {\rm vol}_{n-r} (P_{\gamma^*}\cap S), \quad \varepsilon \to 0.
\end{equation}
\end{thm}

We introduce the remainder $R_\varepsilon (S)$ in the formula \eqref{f:lattice_points-basic} by the formula
\[
R_\varepsilon (S)=n_\varepsilon (S) - \frac{\varepsilon^{-q}}{{\rm
vol}(V^\bot/\Gamma^\bot)}
 \sum_{\gamma^*\in\Gamma_F^*} {\rm vol}_{n-r} (P_{\gamma^*}\cap S).
\]
It is well-known that, in lattice point counting problems, remainder estimates depends heavily on properties of the boundary of the domain.

\begin{thm}\label{mainthm1}
Let $S$ be a bounded set in $E$ such that, for any 
$\gamma^*\in\Gamma^*_F$, the intersection $P_{\gamma^*} \cap S$ is a bounded open subset in $P_{\gamma^*}$ with smooth boundary. Then we have
\[
 R_\varepsilon (S)= O(\varepsilon^{\frac{1}{p-r+1}-q}), \quad \varepsilon \to 0.
\]
\end{thm}

\begin{thm}\label{t:lattice_points}
Let $S$ be a bounded set in $E$ such that, for any 
$\gamma^*\in\Gamma^*_F$, the intersection $P_{\gamma^*} \cap S$ is a bounded open subset in $P_{\gamma^*}$ with smooth boundary.

(1) If, for any $\gamma^*\in\Gamma^*_F$ and $x\in F\cap V^\bot$, the intersection $S\cap \{\gamma^*+x+H\}$ is strictly convex, then 
\[
R_\varepsilon (S)=O(\varepsilon^{\frac{2q}{q+1+2(p-r)}-q}), \quad
\varepsilon \to 0.
\]

(2) If, for any $\gamma^*\in\Gamma^*_F$, the intersection 
$P_{\gamma^*} \cap S$ is strictly convex, then
\[
R_\varepsilon (S)= O(\varepsilon^{\frac{2q}{n-r+1}-q}), \quad
\varepsilon \to 0.
\]
\end{thm}

In the particular case when $E$ is the space ${\mathbb R}^n$
equipped with the standard Euclidean structure and $\Gamma={\mathbb
Z}^n$, the results of this section was obtained in 
\cite{lattice-points,mz-varchenko,varchenko}.

\subsection{Applications to adiabatic limits}\label{s:adiab0}
Applying the results of Section~\ref{s:as} in the situation described in Section~\ref{s:adiab1}, we obtain the asymptotic formula for the eigenvalue distribution function $N_\varepsilon(\lambda)$ of the operator  $\Delta_\varepsilon$ on the flat torus $\mathbb T=\mathcal
E/\Lambda$ in adiabatic limit. We will use notation introduced in Section~\ref{s:adiab1}.

Let $V$ be the linear subspace of $\mathcal E$ spanned by the elements of the free abelian group $\Lambda\cap F$, $r=\dim V$. Then $\Lambda\cap F$ is a lattice in $V$. Let $(\Lambda\cap F)^*$ denote the lattice in $V^*$, dual of $\Lambda\cap F$.

\begin{thm}\label{mainthm2}
For any $\lambda>0$,  the following formula holds as $\varepsilon \to 0$:
\begin{multline}\label{e:Ne}
N_\varepsilon (\lambda) \\ = \varepsilon^{-q} \frac{\omega_{n-r}{\rm
vol}(\mathcal E/\Lambda)}{{\rm vol}(V/\Lambda\cap F)}
 \sum_{\substack{k\in (\Lambda\cap F)^*,\\ |k|_{g^{-1}} < \frac{1}{2\pi}\sqrt{\lambda}}} \left(\frac{\lambda}{4\pi^2}-|k|_{g^{-1}}^2\right)^{(n-r)/2}
+O(\varepsilon^{\frac{2q}{n-r+1}-q}),
\end{multline}
where $\omega_{n-r}$ is the volume of the unit ball in ${\mathbb R}^{n-r}$.
\end{thm}

The expression in the right-hand side of \eqref{e:Ne} was recently studied by J. Lagac\'e and L. Parnovski in \cite{LagaceParnovski}. More precisely, the authors consider the function 
\[
S(\rho,\mathbf k; d,k)=\omega_\ell\sum_{\substack{\gamma\in {\mathbb Z}^k,\\ |\gamma-\mathbf k|<\rho}} (\rho^2-|\gamma-\mathbf k|^2)^{\ell/2}, \quad \rho \in \mathbb R,\quad \mathbf k\in \mathbb R^k
\]
with some natural $k, d$ with $k+\ell=d$, $\ell>0$. This function is treated as the partial density of states of the Laplace operator in $\mathbb R^d$ considered as a $(2\pi\mathbb Z)^d$-periodic operator. For ${\mathbf k}=0$, it can be also interpreted as the integrated density of states of the Laplace operator in $\mathbb T^\ell\times \mathbb R^k$. It is clear that, as $\rho\to \infty$, $S(\rho,\mathbf k; d,k)$ is asymptotically the volume of the ball in $\mathbb  R^d$ of radius $\rho$, $S(\rho,\mathbf k; d,k)\sim \omega_d\rho^d$. In \cite{LagaceParnovski}, upper and lower bounds for the remainder in this asymptotic formula are stated.  These results are then used to compute the integrated density of states of the magnetic Schr\"odinger operator with constant magnetic field.

\subsection{Anomalously small remainder estimates}

In this section, we describe the results on anomalously small remainder estimates in the formula \eqref{f:lattice_points-basic}. These results improve the results of 
\cite{Skriganov89,Skriganov94,Nikichine-Skriganov95,Nikichine-Skriganov98} in the situation under consideration.

Let $\mathbf e=(e_1,\ldots,e_d)$ be an orthonormal basis in a 
$d$-dimensional Euclidean space $\mathcal E$. Put, for $X\in
\mathcal E$,
\[
\operatorname{Nm}_{\mathbf e} X=\prod_{j=1}^d(X,e_j).
\]
Let $K$ be a rectangular parallelepiped in $\mathcal E$ and let 
$\mathbf e=(e_{1},\ldots,e_d)$ be an orthonormal basis in 
$\mathcal E$ such that the edges of $K$ are parallel to the vectors 
$e_{1},\ldots,e_d$. We say that a lattice $L$ in $\mathcal E$ in good position with respect to $K$, if
\[
\inf_{X\in L\setminus\{0\}} \left|\operatorname{Nm}_{\mathbf e}
X\right|>0.
\]

\begin{thm}\label{t:anomaly1}
Let $S$ be a bounded open set in $E$ such that, for any  $\gamma^*\in\Gamma^*_F$, the intersection $P_{\gamma^*} \cap S$ is a rectangular parallelepiped in $P_{\gamma^*}$ with edges parallel either to $F\cap V^\bot$ or to $H$. Suppose that a lattice $\Gamma^\bot$ â $V^\bot$ in good position with respect to
$P_{\gamma^*} \cap S$. Then we have
\[
 R_\varepsilon (S)= O\left(|\log \varepsilon|^{n-r-1}\right), \quad \varepsilon \to 0.
\]
\end{thm}

Theorem \ref{t:anomaly1} is an essential improvement of Theorem 
1.1 from \cite{Skriganov94} for the particular case of anisotropically extending domains under consideration (see Introduction of 
\cite{lattice-points} for weaker results, obtained by a straightforward application of Theorem 1.1 from \cite{Skriganov94}).

It was noted in \cite[Section 7]{Nikichine-Skriganov98} and
\cite{Nikichine-Skriganov95} that the effect of abnormal diminishing of remainder estimates is related with the fact that the domains under consideration are represented as a cartesian product of domains of smaller dimension. We study this effect in the case in question and obtain analogues of the results of 
\cite{Nikichine-Skriganov98,Nikichine-Skriganov95}.

First of all, we recall the definition of an algebraic lattice. Let $K$ be an algebraic number field of degree $d=[K:\mathbb Q]$, having $s$
real embeddings and $t$ pairs of complex-conjugate embeddings. Thus, $d=s+2t$. Let $\{\sigma^\prime_1,\ldots,
\sigma^\prime_s\}$ be a complete system of real embeddings of $K$
in $\mathbb R$ and $\{\sigma^{\prime\prime}_1, \ldots,
\sigma^{\prime\prime}_t\}$ be a complete system of complex embeddings of $K$ in $\mathbb C$. Then the canonical embedding $\sigma$ of
$K$ in ${\mathbb R}^d$ is defined by
\[
\sigma : K\ni \xi \mapsto
\sigma(\xi)=(\sigma^\prime_1(\xi),\ldots,\sigma^\prime_s(\xi),
\sigma^{\prime\prime}_1(\xi),\ldots,\sigma^{\prime\prime}_t(\xi))\in
{\mathbb R}^s\times {\mathbb C}^t\cong {\mathbb R}^d.
\]
If $M\subset K$ is a $\mathbb Z$-module of rank $d$, then the image of  $M$ under $\sigma$,
\[
\Gamma_M=\sigma(M)\subset {\mathbb R}^d,
\]
is a lattice in ${\mathbb R}^d$. Such a lattice is called algebraic. A lattice $\Gamma$ in a $d$-dimensional Euclidean space $V$ is said to be algebraic, if there exists an orthonormal basis $(e_1,\ldots,e_d)$ in $V$ such that the image of $\Gamma$ under the isomorphism $V\cong {\mathbb R}^d$ defined by the basis is an algebraic lattice in ${\mathbb
R}^d$.

Suppose that the subspaces $F\cap V^\bot$ and $H$ are represented as orthogonal direct sums
\begin{equation}\label{e:decomp}
F\cap V^\bot=\bigoplus_{j=1}^{\ell_F} E_j, \quad
H=\bigoplus_{j=\ell_F+1}^{\ell} E_j, \quad \ell_F+\ell_H=\ell,
\end{equation}
where $E_j$ is an Euclidean subspace of $E$ of dimension $m_j$,
$j=1,\ldots,\ell$.

Suppose that the lattice $\Gamma^\bot \subset V^\bot$ is an algebraic lattice, corresponding to a complete $\mathbb
Z$-module $M$ in an algebraic number field $K$ and to an orthonormal basis $(e_1,\ldots,e_{n-r})$ in $V^\bot$.
We say that $\Gamma^\bot$ is in good position with respect to the decomposition \eqref{e:decomp}, if, for each embedding $\sigma_v$, the image $\sigma_v(K)$ of the field $K$ belongs entirely to some subspace $E_j, j=1,\ldots,\ell$.

\begin{thm}\label{t:anomaly2}
Suppose that there is given a decomposition \eqref{e:decomp} of 
$F\cap V^\bot$ and $H$ as orthogonal direct sums and the lattice
$\Gamma^\bot$ is in good position with respect to this decomposition. Let $S\subset E$ be a bounded set such that, for any $\gamma^*\in\Gamma^*_F$, the intersection $P_{\gamma^*}
\cap S$ is an open set in 
$P_{\gamma^*}=\gamma^*+V^{\bot}$ of the form
\[
P_{\gamma^*} \cap S=\prod_{j=1}^\ell (\gamma^*+S_{j,\gamma^*}),
\]
where, for any $j$, the set $S_{j,\gamma^*}$ is a strictly convex open subset of $E_j$.

Then, for any $\delta>0$, there exists s constant 
$C_\delta>0$ such that 
\[
|R_\varepsilon (S)| < C_\delta
\varepsilon^{-\frac{q(n-r-\ell)}{n-r-\ell+2}}|\ln
\varepsilon|^{s+t+\delta}.
\]
\end{thm}

In particular, in the case when $K$ is a totally real algebraic number field, we obtain that $\ell=n-r$, $S_{\gamma^*}$ is a parallelepiped and, for any $\delta>0$,
\[
|R_\varepsilon (S)|<C_\delta |\ln \varepsilon|^{n-r+\delta}.
\]

If $r=p$, we have
\[
\gamma=-q+\frac{2q}{q-\ell+2}=-\frac{q(q-\ell)}{q-\ell+2},
\]
that agrees with the estimate proved in \cite[Theorem
7.1]{Nikichine-Skriganov98}.

The remaining part of the paper is devoted to the proof of the results mentioned above. We use a method developed in 
\cite{lattice-points}, which is based on the Poisson summation formula.
The key point is Lemma~\ref{l:intersection} proved in Section~\ref{s:trivial}, which allows us to reduce our considerations to the case when 
$\Gamma_F$ is trivial. The proof of this lemma relies on some facts about lattices in Euclidean spaces, which are given in Section \ref{s:lattices}. The proofs of Theorems \ref{t:lattice_points-basic}, \ref{mainthm1},
\ref{t:lattice_points} and \ref{mainthm2} are given in Section~\ref{s:t123}, and the proofs of Theorems~\ref{t:anomaly1}
and~\ref{t:anomaly2} in Section~\ref{s:t56}.

\section{Proofs of the main results}
\subsection{Some facts about lattices}\label{s:lattices} In this section, we provide some necessary information about lattices. As above, let $E$ be a finite-dimensional real Euclidean space. Let $L$ be a lattice in $E$. We say that a subspace $X$ in $E$ is an $L$-subspace, if
$X$ is spanned by the elements of $L$. It is clear that, for any $L$-subspace $X$, the set $L(X):=L\cap X$ is a lattice in $X$.

\begin{lem}\label{l:lemX}
For any $L^*$-subspace $X$ in $E$, there is an inclusion
\begin{equation}\label{e:TS}
L \subset \bigsqcup_{\gamma^*\in (L^*(X))^*} (\gamma^*+X^\bot).
\end{equation}
Moreover, for any $\gamma^*\in (L^*(X))^*$, we have
\begin{equation}\label{e:TS2}
L_{\gamma^*}:=L \bigcap (\gamma^*+X^\bot)\neq \emptyset.
\end{equation}
\end{lem}

Here $(L^*(X))^*$ denotes the lattice in $X$, dual of $L^*(X)$:
\[
(L^*(X))^*=\{ \gamma^*\in X :(\gamma^*,L^*(X))\subset\mathbb Z\}.
\]

\begin{proof}
Let $k\in L$. Denote by $\pi_X : E\to X$ the orthogonal projection on $X$. Then, for any $\gamma\in
L^*(X)=L^*\cap X$, we have
\[
(\pi_X(k),\gamma)=(k,\gamma)\in {\mathbb Z}.
\]
Theferore, $\pi_X(k)\in (L^*(X))^*$, that immediately implies 
\eqref{e:TS}.

The relation \eqref{e:TS2} follows from a simple fact that any linear functional $\ell : X\to \mathbb R$ such that $\ell(L^*(X))=\ell(L^*\cap X)\subset \mathbb Z$ can be extended to a linear functional $\tilde\ell : E\to
\mathbb R$ such that $\tilde\ell(L^*)\subset \mathbb Z$.
\end{proof}

\begin{lem}\label{l:ort}
If $X$ is an $L^*$-subspace of $E$, then $X^\bot$ is an 
$L$-subspace of $E$ and
\[
{\rm vol}(X^\bot/L(X^\bot))={\rm
vol}(E/L){\rm vol}(X/L^*(X)).
\]
\end{lem}

\begin{proof}
Let $X$ be an $L^*$-subspace of $E$. Let $(\ell_1, \ldots,
\ell_n)$ is a basis in $L^*$ such that $(\ell_1, \ldots, \ell_r)$ is a basis in  $L^*(X)$. Let $(f_1, \ldots, f_n)$ be the dual basis in $L$. It is easy to see that, for $j=r+1,\ldots,n$, the vector $f_{j}$ belongs to $X^\bot$. Therefore, $X^\bot$ is an $L$-subspace.

Now suppose, as above, that $(\ell_1, \ldots, \ell_r)$ is a basis in 
$L^*(X)$. Denote by $(\ell^*_1, \ldots, \ell^*_r)$ the dual basis in  $(L^*(X))^*$: $(\ell_i,\ell^*_j)=\delta_{ij}$ for any $i,j=1,\ldots,r$. Using \eqref{e:TS2}, for any
$i=1,\ldots,r$, we can choose some $k^*_i\in L$ such that 
$\pi_X(k^*_i)=\ell^*_i$. Let $(k^\bot_1, \ldots, k^\bot_{n-r})$ be a basis in $L(X^\bot)$. Using \eqref{e:TS}, it is easy to show that $(k^*_1,\ldots, k^*_r, k^\bot_1, \ldots, k^\bot_{n-r})$ is a basis in $L$. Therefore, for the volume of the parallelepiped spanned by the vectors $(k^*_1,\ldots, k^*_r,
k^\bot_1, \ldots, k^\bot_{n-r})$, we have the formula:
\[
{\rm vol}_n(k^*_1,\ldots, k^*_r, k^\bot_1, \ldots,
k^\bot_{n-r})={\rm vol}(E/L).
\]
On the other hand, since $k^*_i-\ell^*_i\in X^\bot$ for any
$i=1,\ldots,r$, and the sets $(\ell^*_1,\ldots, \ell^*_r)$ and
$(k^\bot_1, \ldots, k^\bot_{n-r})$ are mutually orthogonal, we get
\begin{align*}
{\rm vol}_n(k^*_1,\ldots, k^*_r, k^\bot_1, \ldots, k^\bot_{n-r})= &
{\rm vol}_n(\ell^*_1,\ldots, \ell^*_r, k^\bot_1, \ldots,
k^\bot_{n-r})\\ = & {\rm vol}_r(\ell^*_1,\ldots, \ell^*_r){\rm
vol}_{n-r}(k^\bot_1, \ldots, k^\bot_{n-r})\\ = & {\rm
vol}(X/(L^*(X))^*){\rm vol}(X^\bot/L(X^\bot)).
\end{align*}
It remains to apply the following well-known identity
\begin{equation}\label{e:vv}
{\rm vol}(X/L^*(X)){\rm vol}(X/(L^*(X))^*)=1.
\end{equation}
\end{proof}

\subsection{Reduction to the case when $\Gamma_F$ is trivial}\label{s:trivial}
In this section, we demonstrate how to reduce our considerations to the case when $\Gamma_F$ is trivial. As above, we assume that $E$ is a finite-dimensional real Euclidean space, $F$ is a linear subspace of $E$, $H=F^{\bot}$ is its orthogonal complement and 
$\Gamma\subset E$ is a lattice in $E$.

Let $V$ be the subspace of $E$ spanned by the elements of 
$\Gamma_F = \Gamma^*\cap F$. Thus, $V$ is a 
$\Gamma^*$-subspace of $E$ and $\Gamma_F = \Gamma^*(V)$.
It is clear that $V$ is a maximal $\Gamma^*$-subspace of $E$, which is contained in $F$.

By Lemma~\ref{l:ort}, the subspace $V^\bot$ is a $\Gamma$-subspace of $E$ and $\Gamma^\bot=\Gamma(V^\bot)$ is a lattice in $V^\bot$. It is easy to see that $V^\bot$ is the minimal $\Gamma$-subspace of $E$, which contains $H$.

First of all, observe that, thanks to \eqref{e:TS}, we have
\[
\Gamma \subset \bigsqcup_{\gamma^*\in \Gamma_F^*} P_{\gamma^*},
\]
moreover,  $\Gamma\cap P_{\gamma^*} \neq \emptyset$ for any
$\gamma^*\in \Gamma_F^*$. Therefore, for any bounded subset $S$ of $E$, we have the identity
\[
n_\varepsilon (S)=\sum_{\gamma^*\in \Gamma^*_F}\#
(T_\varepsilon(S)\cap (\Gamma\cap P_{\gamma^*})),
\]
where the sum in the right-hand side has finite number of non-vanishing terms.

The affine subspace $P_{\gamma^*}$ is identified with the linear space  $V^\bot$ by the map
\[
U_{\gamma^*} : v\in V^\bot \mapsto \gamma^*+v \in P_{\gamma^*}.
\]
The space $P_{\gamma^*}$ is invariant under the map 
$T_\varepsilon$, and the restriction of $T_\varepsilon$ to
$P_{\gamma^*}$ corresponds under the isomorphism $U_{\gamma^*}$ to the map $T^\bot_\varepsilon$ of $V^\bot$ defined by \eqref{e:Te} for the decomposition $V^\bot=(F\cap V^\bot)\oplus H$.

For any $\gamma^*\in \Gamma_F^*$, we fix an arbitrary element 
$k_{\gamma^*}\in \Gamma\cap P_{\gamma^*}$. Then
\[
\Gamma\cap P_{\gamma^*}=k_{\gamma^*}+\Gamma^\bot.
\]

Using these identifications, we obtain
\[
\# (T_\varepsilon(S)\cap (\Gamma\cap P_{\gamma^*}))=\#
((T^\bot_\varepsilon(S\cap P_{\gamma^*})-k_{\gamma^*})\cap
\Gamma^\bot),
\]
where, in the right-hand side, $S\cap P_{\gamma^*}$ is considered as an open subset of $V^\bot$. Therefore, if we introduce the notation 
\[
n_\varepsilon (S\cap P_{\gamma^*},-k_{\gamma^*})=\#
((T^\bot_\varepsilon(S\cap P_{\gamma^*})-k_{\gamma^*})\cap
\Gamma^\bot)
\]
we obtain finally that
\begin{equation}\label{e:ne-sum}
n_\varepsilon (S)=\sum_{\gamma^*\in \Gamma^*_F}n_\varepsilon (S\cap
P_{\gamma^*},-k_{\gamma^*}),
\end{equation}

An important fact about a relative position of the lattice
$\Gamma^\bot $ and the subspace $F\cap V^\bot$ of $V^\bot$ is given by the following lemma.

\begin{lem}\label{l:intersection}
We have
\[
{\Gamma^\bot}^*\cap (F\cap V^\bot)=\{0\},
\]
where ${\Gamma^\bot}^*\subset V^{\bot}$ is the dual lattice of 
$\Gamma^\bot $.
\end{lem}

\begin{proof}
By Lemma~\ref{l:lemX} applied to the lattice $L=\Gamma^*$ and the
$\Gamma^*$-subspace $X=V^\bot$ we obtain that, for any
$\gamma^*\in {\Gamma^\bot}^*$,
\[
\Gamma^* \bigcap (\gamma^*+V)\neq \emptyset.
\]
Thus, given $\gamma^*\in {\Gamma^\bot}^*\cap F$, there exists $v\in V$ such that $\gamma^*+v\in \Gamma^*\cap
F=\Gamma_F$. Hence, $\gamma^*+v\in V$, and, therefore,
$\gamma^*\in V$. Since $\gamma^*\in V^\bot$, we get
$\gamma^*=0$.
\end{proof}

The formula \eqref{e:ne-sum} along with Lemma \ref{l:intersection} allow us to reduce our considerations to the case when $\Gamma_F$ is trivial. From now and later on, we will assume that 
\begin{equation}\label{e:ass}
\Gamma_F= \Gamma^*\cap F=\{0\}.
\end{equation}

\subsection{Asymptotic formula and remainder estimates}\label{s:t123}
The proofs of Theorems \ref{t:lattice_points-basic}, \ref{mainthm1}
and~\ref{t:lattice_points} under Condition \eqref{e:ass} make use of a method developed in \cite{lattice-points} and based on the Poisson summation formula. We describe briefly this method.

As above, let $F$ be a linear subspace of $E$ and 
$\Gamma\subset E$ be a lattice in $E$. Let $S$ be a bounded open subset of $E$. For any $\varepsilon>0$ and $v\in E$, we denote by $n_\varepsilon (S,v)$ the number of points of $\Gamma$ in the domain
$T_\varepsilon(S)+v$:
\begin{equation}\label{e:nSv}
n_\varepsilon (S,v)=\# ((T_\varepsilon(S)+v)\cap \Gamma).
\end{equation}
Under Condition \eqref{e:ass}, the asymptotic formula 
\eqref{f:lattice_points-basic} takes the form:
\begin{equation}\label{f:lattice_points-basic0}
n_\varepsilon (S,v) \sim \frac{\varepsilon^{-q}}{{\rm
vol}(E/\Gamma)} {\rm vol}_{n}(S), \quad \varepsilon \to 0.
\end{equation}
Therefore, the initial problem is reduced to the study of the remainder  $R_\varepsilon (S,v)$ in the formula~\eqref{f:lattice_points-basic0} given by
\begin{equation}\label{eq:Rv}
R_\varepsilon (S,v)=n_\varepsilon (S,v)
-\frac{\varepsilon^{-q}}{{\rm vol}(E/\Gamma)} {\rm vol}_{n}(S).
\end{equation}

Let $\chi_{S}$ be the indicator of $S$. It is easy to see that
\[
n_\varepsilon (S,v)=\sum_{\gamma\in
\Gamma}\chi_{T_\varepsilon(S)+v}(\gamma) = \sum_{\gamma\in \Gamma}
\chi_{S} (T_{\varepsilon^{-1}}(\gamma-v)).
\]

We will write the representation of a vector $x\in E$, corresponding to the decomposition $E=F\bigoplus H$, in the form:
\[
x=x_F+x_H, \quad x_F\in F, x_H\in H.
\]
Observe that
\[
T_{\varepsilon}(x)=x_F +\varepsilon^{-1}x_H.
\]

Let $\rho_{1,1} \in C^\infty_c(E)$ be a non-negative function such that ${\rm supp}\,\rho_{1,1}\subset B(0,1)=\{x\in E :
|x|\leq 1\}$ and $\int_{E} \rho_{1,1}(x)\,dx=1$. For any $t_F>0$ and
$t_H>0$, introduce a function $\rho_{t_F,t_H} \in C^\infty_c(E)$ by the formula
\begin{equation}\label{e:rho}
\rho_{t_F,t_H}(x)=\frac{1}{t_F^{p-r}t_H^{q}}\rho(t^{-1}_Fx_F+t^{-1}_Hx_H),
\quad x=x_F+x_H\in E.
\end{equation}
The function $\rho_{t_F,t_H}$ is supported in the ellipsoid
\[
B(0,t_F,t_H)=\left\{x\in E :
\frac{|x_F|^2}{t^2_F}+\frac{|x_H|^2}{t^2_H}<1\right\}.
\]

Define a function $n_{\varepsilon,t_F,t_H}(S,v)$ by
\[
n_{\varepsilon,t_F,t_H}(S,v)=\sum_{\gamma\in
\Gamma}(\chi_{T_\varepsilon(S)+v}\ast \rho_{t_F,t_H})(k),
\]
where $\chi_{T_\varepsilon(S)+v} \ast \rho_{t_F,t_H}\in
C^\infty_0(E)$ is defined by
\[
(\chi_{T_\varepsilon(S)+v} \ast \rho_{t_F,t_H})(y) =\int_{E}
\chi_{T_\varepsilon(S)+v}(y-x) \rho_{t_F,t_H}(x)\, dx, \quad y\in E.
\]

For any domain $S\subset E$ and for any $t_F>0$ and $t_H>0$,
let
\[
S_{t_F,t_H}=\bigcup_{x\in S}(x+B(0,t_F,t_H)),
\]
and
\[
S_{-t_F,-t_H} = E \setminus (E \setminus S)_{t_F,t_H}.
\]
It is easy to see that, for any $\varepsilon>0$, $t_F>0$ and $t_H>0$, there is a relation
\[
T_\varepsilon(S_{t_F,\varepsilon t_H})=(T_\varepsilon(S))_{t_F,t_H}.
\]

\begin{lem}[cf. Lemma 7 in \cite{lattice-points}]
For any $\varepsilon>0$, $t_F>0$ and $t_H>0$, the following inequalities hold:
\[
n_{\varepsilon,t_F,t_H}(S_{-t_F,-\varepsilon t_H},v) \leq
n_\varepsilon (S,v) \leq n_{\varepsilon,t_F,t_H}(S_{t_F,\varepsilon
t_H},v).
\]
\end{lem}

For any function $f$ from Schwartz space ${\mathcal S}(E)$, we define its Fourier transform $\hat{f}\in {\mathcal S}(E)$ by
\[
\hat{f}(\xi)=\int_{E}e^{-2\pi i(\xi,x)}f(x)\,dx.
\]

Recall the Poisson summation formula 
\begin{equation}\label{e:Poisson}
\sum_{k\in \Gamma} f(k)=\frac{1}{{\rm vol}(E/\Gamma)}\sum_{k^*\in
\Gamma^*}\hat{f}(k^*), \quad f\in {\mathcal S}(E),
\end{equation}
where $\Gamma^*\subset E$ is the dual lattice of $\Gamma$.
Applying the formula \eqref{e:Poisson} to the function
\begin{equation}\label{e:f}
f(x)=(\chi_{T_\varepsilon(S_{t_F,\varepsilon t_H})+v} \ast
\rho_{t_F,t_H})(x), \quad x\in E,
\end{equation}
and using the relations
\begin{gather*}
\hat\chi_{T_\varepsilon(S_{t_F,\varepsilon t_H})+v}(\xi)=
\varepsilon^{-q}e^{-2\pi i(\xi,v)} \hat\chi_{S_{t_F,\varepsilon
t_H}}(T_{\varepsilon}(\xi)), \quad
\xi\in E, \\
\hat\rho_{t_F,t_H}(\xi)=\hat\rho_{1,1}(t_F\xi_F+t_H\xi_H), \quad
\xi\in E,
\end{gather*}
we get
\begin{multline}\label{e:ne1}
n_{\varepsilon,t_F,t_H}(S_{t_F,\varepsilon t_H},v)
\\ = \frac{\varepsilon^{-q}}{{\rm vol}(E/\Gamma)}
\sum_{k\in \Gamma^*} e^{-2\pi i(k,v)} {\hat\chi}_{S_{t_F,
\varepsilon t_H}}( T_{\varepsilon}(k))\hat\rho_{1,1}(t_Fk_F+t_Hk_H).
\end{multline}

We write
\[
n_{\varepsilon,t_F,t_H}(S_{t_F,\varepsilon
t_H},v)=n^{\prime}_{\varepsilon,t_F,t_H} (S_{t_F,\varepsilon
t_H},v)+n^{\prime\prime}_{\varepsilon,t_F,t_H} (S_{t_F,\varepsilon
t_H},v),
\]
where
\begin{multline*}
n^{\prime}_{\varepsilon,t_F,t_H} (S_{t_F,\varepsilon t_H},v)\\
=\frac{\varepsilon^{-q}}{{\rm vol}(E/\Gamma)} \sum_{k\in \Gamma^*,
k_H= 0} e^{-2\pi i(k,v)} {\hat\chi}_{S_{t_F,\varepsilon t_H}}(
T_{\varepsilon}(k))\hat\rho_{1,1}(t_Fk_F+t_Hk_H)
\end{multline*}
and
\begin{multline*}
n^{\prime\prime}_{\varepsilon,t_F,t_H}
(S_{t_F,\varepsilon t_H},v)\\
=\frac{\varepsilon^{-q}}{{\rm vol}(E/\Gamma)} \sum_{k\in \Gamma^*,
k_H\neq 0} e^{-2\pi i(k,v)} {\hat\chi}_{S_{t_F,\varepsilon t_H}}(
T_{\varepsilon}(k)) \hat\rho_{1,1}(t_Fk_F+t_Hk_H).
\end{multline*}

Let $k\in \Gamma^*$ be such that $k_H=0$. Then $k\in \Gamma^*\cap
F$, and, by \eqref{e:ass}, we obtain that $k=0$. It is in this place that we heavily use Condition \eqref{e:ass} and, thereby, Lemma \ref{l:intersection}. Thus, we get
\[
n^{\prime}_{\varepsilon,t_F,t_H} (S_{t_F,\varepsilon t_H},v) =
\frac{\varepsilon^{-q}}{{\rm vol}(E/\Gamma)}
\hat\chi_{S_{t_F,\varepsilon t_H}} (0) =
\frac{\varepsilon^{-q}}{{\rm vol}(E/\Gamma)}{\rm vol}_{n}
(S_{t_F,\varepsilon t_H}).
\]
Using the estimate
\[
{\rm vol}_{n} (S_{t_F,\varepsilon t_H} \setminus S)\leq C
(t_F+t_H\varepsilon),
\]
we obtain
\[
|n^{\prime}_{\varepsilon,t_F,t_H} (S_{t_F,\varepsilon t_H},v)-
\frac{\varepsilon^{-q}}{{\rm vol}(E/\Gamma)}{\rm vol}_{n} (S)|\leq
C (t_F+t_H\varepsilon).
\]
Thus, for the remainder $R_\varepsilon (S,v)$ in the formula~\eqref{f:lattice_points-basic0} given by~\eqref{eq:Rv}, we derive the estimate
\begin{equation}\label{eq:R}
|R_\varepsilon (S,v)|\leq \frac{\varepsilon^{-q}}{{\rm
vol}(E/\Gamma)}\Big( C (t_F+t_H\varepsilon) + |R^+_\varepsilon
(t_F,t_H)| + |R^-_\varepsilon (t_F,t_H)|\Big),
\end{equation}
where
\[
R^\pm_\varepsilon (t_F,t_H) =\sum_{k\in \Gamma^*, k_H\neq 0}e^{-2\pi
i(k,v)} {\hat\chi}_{S_{\pm t_F, \pm \varepsilon t_H}}(
T_{\varepsilon}(k)) \hat\rho_{1,1}(t_Fk_F+t_Hk_H).
\]

To complete the proofs of Theorems
\ref{t:lattice_points-basic}, \ref{mainthm1}
and~\ref{t:lattice_points}, it remains to estimate 
$R^\pm_\varepsilon (t_F,t_H)$. For this, we can use well-known estimates of the Fourier transform of the indicator of a bounded domain with smooth or strictly convex boundary. In the case when $E$ is the space ${\mathbb R}^n$ equipped with the standard Euclidean structure and $\Gamma={\mathbb Z}^n$, this is done in \cite{lattice-points}, \cite{varchenko}. The arguments of these papers can be easily extended to the case under consideration.

For the proof of Theorem~\ref{mainthm2}, we can use the formula  \eqref{e:Ne} and apply Theorem~\ref{t:lattice_points} in the situation when $E=\mathcal E^*$, $\Gamma=\Lambda^*$ and $S$ is the ball
$B_{\sqrt{\lambda}}(0)$. We only observe the identity
\[
{\rm vol}(V^\bot/\Gamma^\bot)={\rm vol}(\mathcal E/\Lambda^*){\rm
vol}(V/\Lambda\cap F)=\frac{{\rm vol}(V/\Lambda\cap F)}{ {\rm
vol}(\mathcal E/\Lambda)},
\]
which follows from Lemma \ref{l:ort} and \eqref{e:vv}.

\subsection{Anomalously small remainder estimates}\label{s:t56} This section is devoted to the proof of Theorems~\ref{t:anomaly1} and~\ref{t:anomaly2}. If Condition \eqref{e:ass} holds, these theorems are direct consequences of the results of \cite{Skriganov94,Nikichine-Skriganov98}.

Let us begin with the proof of Theorem~\ref{t:anomaly1}. Suppose that Condition \eqref{e:ass} holds. Let $S$ be a rectangular parallelepiped
in $E$ with edges parallel to vectors $e_{1},\ldots,e_{n}$ in $E$ such that
$(e_{1},\ldots,e_{p})$ is an orthonormal basis in $F$ and 
$(e_{p+1},\ldots,e_{n})$  is an orthonormal basis in $H$.
The orthonormal basis $(e_{1},\ldots,e_{n})$ gives rise to an isomorphism of Euclidean spaces $E\cong \mathbb R^n$, which sends the parallelepiped $S$ to the parallelepiped 
$S_0=\prod_{j=1}^n[a_j, b_j]$ and the lattice $\Gamma$ to
a lattice $\Gamma_0$ in $\mathbb R^n$, which is in good position with respect to $S_0$ (admissible in the sense of \cite{Skriganov94}). Under this isomorphism, the parallelepiped $T_\varepsilon(S)+v$ goes to the parallelepiped
\[
T_\varepsilon(S_0)+v=\prod_{j=1}^p[a_j+v_j, b_j+v_j]
\prod_{j=p+1}^n[\varepsilon^{-1}a_j+v_j, \varepsilon^{-1}b_j+v_j],
\]
where $v=\sum_{j=1}^nv_je_j$. It can be represented in the form 
\[
T_\varepsilon(S_0)+v=T\cdot \mathbb K^n+X,
\]
where $\mathbb K^n=\prod_{j=1}^n[-\frac12, \frac12]$ is the unit cube in $\mathbb R^n$,
\[
T=(t_1,\ldots,t_n), \quad t_j=\begin{cases} b_j-a_j, & j=1,\ldots,p,\\
\varepsilon^{-1}(b_j-a_j), & j=p+1,\ldots,n,
\end{cases}
\]
and
\[
X=(X_1,\ldots,X_n), \quad X_j=\begin{cases} \frac{b_j+a_j}{2}+v_j, & j=1,\ldots,p,\\
\varepsilon^{-1}\frac{b_j+a_j}{2}+v_j, & j=p+1,\ldots,n.
\end{cases}
\]
Therefore, one can apply Theorem 1.1 from \cite{Skriganov94},
which implies that
\[
R_\varepsilon (S,v)<c[\log (2+|\operatorname{Nm}T|)]^{n-1}<C|\log
\varepsilon|^{n-1}.
\]

The proof of Theorem~\ref{t:anomaly2} when Condition \eqref{e:ass} holds is a direct consequence of the following result (see
Theorem~\ref{t:NS71} below), which is a minor generalization of Theorem 7.1 in \cite{Nikichine-Skriganov98}.

Let us consider a finite-dimensional Euclidean space $E$ represented as a orthogonal direct sum 
\begin{equation}\label{e:decomp2}
E=\bigoplus_{j=1}^{\ell} E_j,
\end{equation}
where $E_j$ is a subspace in $E$ of dimension $m_j$,
$j=1,\ldots,\ell$.
Suppose that a lattice $\Gamma \subset E$ is an algebraic lattice in $E$, corresponding to a complete $\mathbb Z$-module $M$ in an algebraic number field $K$ and an orthonormal basis $(e_1,\ldots,e_{n})$ in $E$.
Then we have an isomorphism 
\[
E\cong {\mathbb R}^{s}\times {\mathbb C}^{t},\quad s+2t=n,
\]
which defines a norm function ${\rm Nm}_{E}$ in
$E$: for any $X\in E$, corresponding under the above isomorphism to  $(x^\prime_1,\ldots,x^\prime_{s},
x^{\prime\prime}_1,\ldots, x^{\prime\prime}_{t})\in {\mathbb
R}^{s}\times {\mathbb C}^{t}$,
\[
\operatorname{Nm}_{E} X=\prod_{\alpha=1}^{s}|x^\prime_\alpha |
\prod_{\beta=1}^{t}|x^{\prime\prime}_\beta |^2.
\]

We define the product of vectors $X=(x^\prime_1,\ldots,x^\prime_{s},
x^{\prime\prime}_1,\ldots, x^{\prime\prime}_{t})\in {\mathbb
R}^{s}\times {\mathbb C}^{t}$ and $Y=(y^\prime_1,\ldots,y^\prime_{s},
y^{\prime\prime}_1,\ldots, y^{\prime\prime}_{t})\in {\mathbb
R}^{s}\times {\mathbb C}^{t}$ as follows: \[X\cdot Y=
(x^\prime_1y^\prime_1,\ldots,x^\prime_{s}y^\prime_{s},
x^{\prime\prime}_1y^{\prime\prime}_1,\ldots,
x^{\prime\prime}_{t}y^{\prime\prime}_{t})\in {\mathbb R}^{s}\times
{\mathbb C}^{t}.\] It is easy to see that $\operatorname{Nm}_{E} X\cdot Y
=\operatorname{Nm}_{E} X\operatorname{Nm}_{E}Y$.

For an open set $S$ in $E\cong {\mathbb R}^{s}\times {\mathbb
C}^{t}$ and a vector $T\in {\mathbb R}^{s}\times {\mathbb C}^{t}$, denote by $T\cdot S$ the set obtained by the multiplying each point of $S$ by $T$, $\operatorname{Nm} T\neq 0$. Observe that
${\rm vol}\,(T\cdot S)=|\operatorname{Nm} T| {\rm vol}\,(S)$

Suppose that a lattice $\Gamma$ is in good position with respect to the decomposition \eqref{e:decomp2}, that is,
for each embedding $\sigma_v$, the image $\sigma_v(K)$ of $K$ is entirely contained in one of subspaces $E_j, j=1,\ldots,\ell$. Then, for any $j=1,\ldots,\ell$, the subspace $E_j$ is invariant under the multlitiplication by $T\in {\mathbb R}^{s}\times {\mathbb C}^{t}$. In particular, there is a natural isomorphism
\[
E_j\cong {\mathbb R}^{s_j}\times {\mathbb C}^{t_j},\quad
s_j+2t_j=m_j.
\]

Let $S$ be a bounded open subset in $E$. For any
$T\in {\mathbb R}^{s}\times {\mathbb C}^{t}$ and $v\in E$, we denote by $n_T (S,v)$ the number of points of $\Gamma$ in the domain $T\cdot
S+v$:
\begin{equation}\label{e:nST}
n_T (S,v):=\# ((T\cdot S)+v)\cap \Gamma),
\end{equation}
and by $R_T (S,v)$ the remainder in the asymptotic formula for $n_T (S,v)$:
\[
R_T (S,v):=n_T (S,v) -\frac{1}{{\rm vol}(E/\Gamma)} {\rm
vol}\,(T\cdot S)=n_T (S,v) -\frac{|\operatorname{Nm} T|}{{\rm
vol}(E/\Gamma)} {\rm vol}\,(S).
\]

Assume that the set $S\subset E$ is a cartesian product  $S=\prod_{j=1}^\ell S_{j}$, where, for any $j$, the set $S_{j}$ is a strictly convex bounded open subset in $E_j$.

\begin{thm}\label{t:NS71} For any $\delta>0$, the following estimate holds:
\[ |R_T (S,v)|<C_\delta |\operatorname{Nm}
T|^{\frac{n-\ell}{n-\ell+2}}\ln^{s+t+\delta}(2+|\operatorname{Nm}
T|).
\]
\end{thm}

The proof of Theorem \ref{t:NS71} is obtained by a straightforward generalization of the proof of Theorem
2.2 in \cite{Nikichine-Skriganov98} (see also the proof of Theorem 7.1 in \cite{Nikichine-Skriganov98}).

Now the proof of Theorem~\ref{t:anomaly2} is completed as follows. Suppose that Condition \eqref{e:ass} is valid. In the case under consideration, the operator $T_\varepsilon$ is represented as  $T_\varepsilon(X)=T_\varepsilon\cdot X$ with some
$T_\varepsilon\in {\mathbb R}^{s}\times {\mathbb C}^{t}$, moreover,
$\operatorname{Nm} T_\varepsilon = \varepsilon^{-q}$. Therefore, Theorem~\ref{t:anomaly2} follows immediately from Theorem~\ref{t:NS71}:
\[
|R_{T_\varepsilon} (S,v)|=\left|n_{\varepsilon} (S,v)
-\frac{\varepsilon^{-q}}{{\rm vol}(E/\Gamma)} {\rm vol}\,(S)\right|
<\varepsilon^{-\frac{q(n-r-\ell)}{n-r-\ell+2}}|\ln
\varepsilon|^{s+t+\delta}.
\]

\end{document}